\documentclass[11pt]{amsart}

\usepackage{amsmath,amssymb,graphics,amsthm,amscd}
\usepackage{longtable}
\usepackage{cite,xypic}
\usepackage{color}
\usepackage{graphicx}
\usepackage{url}
\usepackage{epsf}
\usepackage{fancyhdr}
\usepackage{lscape}
  \renewenvironment{thebibliography}[1]{
    \begin{oldthebibliography}{#1}
      \setlength{\parskip}{0ex}
      \setlength{\itemsep}{0ex}
  }
  {
    \end{oldthebibliography}
  }
  
\setlongtables

\newcommand{\G}{\mathcal{G}}

\DeclareMathOperator{\Hom}{Hom}
\DeclareMathOperator{\Ext}{Ext}

\newcommand{\NN}{\mathbb{N}}

\DeclareMathOperator{\Ind}{Ind}

\newcommand{\opH}{\text{\rm H}}

\newcommand{\la}{\langle}
\newcommand{\ra}{\rangle}

\begin{document}

\newcounter{rownum}
\setcounter{rownum}{0}

\newtheorem{lemma}{Lemma}[section]
\newtheorem{theorem}[lemma]{Theorem}
\newtheorem*{TA}{Theorem A}
\newtheorem*{TB}{Theorem B}
\newtheorem*{TC}{Theorem C}
\newtheorem*{C3}{Corollary 3}
\newtheorem*{T4}{Theorem 4}
\newtheorem*{C5}{Corollary 5}
\newtheorem*{claim}{Claim}
\newtheorem{corollary}[lemma]{Corollary}
\newtheorem{conjecture}[lemma]{Conjecture}
\newtheorem{prop}[lemma]{Proposition}
\theoremstyle{remark}
\newtheorem{remark}[lemma]{Remark}
\newtheorem{obs}[lemma]{Observation}
\theoremstyle{definition}
\newtheorem{defn}[lemma]{Definition}

  \def\hal{\unskip\nobreak\hfil\penalty50\hskip10pt\hbox{}\nobreak
  \hfill\vrule height 5pt width 6pt depth 1pt\par\vskip 2mm}

\newenvironment{changemargin}[1]{%
  \begin{list}{}{%
    \setlength{\topsep}{0pt}%
    \setlength{\topmargin}{#1}%
    \setlength{\listparindent}{\parindent}%
    \setlength{\itemindent}{\parindent}%
    \setlength{\parsep}{\parskip}%
  }%
  \item[]}{\end{list}}

\parindent=0pt
\addtolength{\parskip}{0.5\baselineskip}

 \title[Ree analogue]{On extensions for Ree groups of type $F_4$}

\author{David I. Stewart}
\address{New College, Oxford\\ OX1 3BN, UK}
\email{david.stewart@new.ox.ac.uk {\text{\rm(Stewart)}}}

\pagestyle{plain}
\begin{abstract}Let $k$ be an algebraically closed field of characteristic $p=2$. Let $G=F_4$ be simply connected over $k$ and let $\sigma:G\to G$ be an endomorphism such that the fixed point set $G(\sigma)$ is a Ree group. We show, using the methods of \cite{BNP06}, that self-extensions of simple $kG(\sigma)$ modules vanish generically and that for all but the first few Ree groups of type $F_4$, the $1$-cohomology for $G(\sigma)$ with coefficients in a simple $kG$-module can be identified with the $1$-cohomology for $G$ with coefficients in a (possibly different) simple $G$-module. \end{abstract}
\maketitle
\section{Introduction}
Let $G$ be a simple algebraic group over an algebraically closed field $k$ of characteristic $p>0$ and let $\sigma:G\to G$ be a surjective endomorphism of $G$ such that the fixed point set $G(\sigma):=G(k)^\sigma$ is a finite group. Then $G(\sigma)$ is in fact a finite group of Lie type. In \cite{BNP06} the authors investigate the connection between the $1$-cohomology for $G$ and $G(\sigma)$ in case $G(\sigma)$ is a Chevalley or Steinberg group. They exclude from consideration the Ree and Suzuki groups. As the authors explain, since their results involve comparing extensions between simple modules, there was little point dealing with the Ree groups of type $G_2$ or the Suzuki groups, since the exact computation of these extensions had been completed in papers of Sin: \cite{Sin92} and \cite{Sin93}. However there remains a gap, since there are no general results for the Ree groups of type $F_4$. This paper exists to fill that gap with results in the flavour of \cite{BNP06}.

There are three theorems in [\emph{ibid.}] for which we wish to create analogues for the Ree groups $^2F_4(q)$. They are \begin{enumerate}\item {[\emph{ibid.},Thm.~5.4]} which shows that, generically, self-extensions between simple $kG(\sigma)$ modules vanish.
\item {[\emph{ibid.},Thm.~5.5]} which shows that with finitely many exceptions depending only on the root system $\Phi(G)$, the $1$-cohomology $\opH^1(G(\sigma),L)$ of a simple $kG(\sigma)$-modules $L$ can be identified with the $1$-cohomology $\opH^1(G,M)$ of a possibly different simple $G$-module $M$.
\item {[\emph{ibid.},Thm.~5.6]} which is an analogue of the above for degree $1$ $kG(\sigma)$-extensions between simple modules.\end{enumerate}

The analogues of these will be our theorems \ref{self} and \ref{ext} below. (The latter incorporates (ii) as a special case of (iii).) We will not give the exact statements in this introduction. 

It is unsurprising that some analogue must exist and our proofs differ hardly at all from those in \cite{BNP06} (though we are able to abbreviate significantly using some new technology from \cite{BNPPSS}). The main industry here is establishing out how to reduce the various components of the many spectral sequence calculations to those involving only the classical Frobenius kernels `$G_r$' with $r\in \mathbb N$.

\subsection*{Acknowledgements}Many thanks to C.~ Pillen for discussions.

\section{Notation}
Assume the notation in \cite{Jan03}. We supplant as necessary with notation from \cite{BNPPSS}, given below. The reader is warned that the notation in the other papers cited is more-or-less pair-wise inconsistent.

For the remainder of the paper, let $G$ be a group of type $F_4$ defined over an algebraically closed field $k$ of characteristic $p=2$. Let $\tau$ be the special isogeny $G\to G$ satisfying $\tau^2=F$, the Frobenius endomorphism on $G$. We denote the scheme-theoretic kernel of $\tau$ by $G_{1/2}$; for an even integer $r=2s$, we denote the kernel $G_\sigma$ of $\sigma=F^{s}$ by $G_{r/2}=G_{s}$ and if $r=2s+1$ is odd, the kernel $G_\sigma$ of $\sigma=\tau\circ F^s$ by $G_{r/2}=G_{s+1/2}$. Also when $r$ is odd, denote by $G(\sigma)=G(r/2)=G(s+1/2)$ the group of fixed points $G(k)^{\tau\circ F^s}$. The finite group $G(s+1/2)$ is a Ree group $^2F_4(q)$. If $V$ is a $G$-module, denote by $V^{[\sigma]}=V^{[r/2]}$ the twist of the module by precomposing the action map with one of the $G$-endomorphisms $\sigma=F^s$ or $\sigma=\tau\circ F^s$ according as $r$ is even or odd. Similarly, if $G_\sigma=G_{r/2}$ acts trivially on $V$, we may define an untwist $V^{[\sigma^{-1}]}=V^{[-r/2]}$.

Let $X^+$ be the set of dominant weights of $G$ and  denote by $X_\sigma$ the subset of $X^+$ defined by weights which are $\sigma$-restricted. In case $r$ is even, $X_\sigma=X_{r/2}$; when $r=2s+1$ is odd, then the condition that $\lambda\in X^+$ be $\sigma$-restricted is that $\langle\lambda, \alpha^\vee\rangle < p^{s+1}$ for $\alpha\in\Pi$ short, and $< p^s$ in case $\alpha\in\Pi$ is long, where $\Pi$ denotes the set of simple roots of $G$. Any dominant weight $\lambda$ can be uniquely written as $\lambda+\sigma^*\lambda_1$, where $\lambda_0$ is $\sigma$-restricted and $\lambda_1\in X^+$. Here $\sigma^* : X \to X$ is the restriction to $X \subset k[T]$ of the comorphism $\sigma^*$ of $\sigma$. Then $L(\lambda) \cong L(\lambda_0) \otimes L(\lambda_1)^{[\sigma]}$, an analogue of SteinbergÕs tensor product theorem.

We make significant use of spectral sequences \[E_2^{ij}:=\Ext^i_{H_1/H_2}(k,\Ext_{H_2}^j(M,N))\Rightarrow \Ext_{H_1}^{i+j}(M,N),\]where $H_2=G_{s/2}$ for $s\in\mathbb N$ and $H_1$ is either $G$ or $G_{r/2}$ with $r\geq s$. See \cite[Remark 2.2.1(a)]{BNPPSS} for a discussion.

\section{Results}

Set $\G(k)=\Ind_{G(\sigma)}^G$. This is an infinite dimensional module as the coset space $G/G(\sigma)$ is affine. 

\begin{theorem}\label{filtration}The $G$-module $\G(k)$ has a filtration with sections $H^0(\lambda)\otimes H^0(\lambda^*)^{[\sigma]}$, one for each $\lambda\in X^+$ and occurring in an order compatible with the dominance order on $X^+$.\end{theorem}
\begin{proof}This is proved in \cite[Prop.~3.1.2]{BNPPSS} and the remarks at the beginning of [{\it ibid.}, \S3.2].\end{proof}

Following \cite[\S2.4]{BNPtw}, define $\G_\Omega(k)$ as a certain truncation of $\G(k)$; specifically it is the largest submodule of $\G(k)$ whose high weights are in the set $\{\nu\in X^+|\langle\nu,\alpha_0^\vee\rangle<\langle 2(\sigma^*-1)\rho,\alpha_0^\vee\rangle\}$. By the above theorem, $\G_\Omega(k)$ also has sections $H^0(\lambda)\otimes H^0(\lambda^*)^{[\sigma]}$. Using $\langle\rho,\alpha_0^\vee\rangle=h-1$, and that $\la\lambda,\alpha_0^\vee\ra<\la\mu,\alpha_0^\vee\ra\iff\la\sigma^*\lambda,\alpha_0^\vee\ra<\la\sigma^*\mu,\alpha_0^\vee\ra$, we see that $\G_\Omega(k)$ has only $|\Gamma|$ sections,  one for each $\lambda\in \Gamma$ where $\Gamma=\{\nu\in X^+|\langle\nu,\alpha_0^\vee\rangle< 2(h-1)\}$.

\begin{theorem}\label{isom}For any $\lambda,\mu\in X_\sigma$, we have \[\Ext^1_{G(\sigma)}(L(\lambda),L(\mu))\cong \Ext^1_{G}(L(\lambda),L(\mu)\otimes \G_\Omega(k)).\]\end{theorem}
\begin{proof} This is \cite[Thm.~2.4]{BNPtw}\end{proof}

The following lemma will be used in various spectral sequence calculations.

\begin{lemma}\label{lemma} Let $r\in \NN$ and set $s=[r/2]$. Let $\lambda,\mu\in X_{r/2}$ and let $\nu\in X^+$ satisfy $\langle\nu,\alpha_0^\vee\rangle< p^s$. Then the $G$-module $\Hom_{G_{r/2}}(L(\lambda),L(\mu)\otimes H^0(\nu))$ has trivial $G$-structure; that is, it is isomorphic to $\Hom_{G}(L(\lambda),L(\mu)\otimes H^0(\nu))$\end{lemma}

\begin{proof}When $r$ is even this is \cite[Prop.~3.1]{BNP06}---we follow the same argument when $r$ is odd. Dualising if necessary, we may assume $\langle\mu,\alpha_0^\vee\rangle\leq \langle\lambda,\alpha_0^\vee\rangle$. All $G$-composition factors of $\Hom_{G_\sigma}(L(\lambda),L(\mu)\otimes H^0(\nu))$ are $G_\sigma$-trivial so must be of the form $L(\gamma)^{[\sigma]}$. For such a factor, $\lambda+\sigma^*\gamma$ is a weight of $L(\mu)\otimes H^0(\nu)$ and hence \[\la\lambda+\sigma^*\gamma,\alpha_0^\vee\ra\leq \la\mu+\nu,\alpha_0^\vee\ra\leq\la\lambda,\alpha_0^\vee\ra+\la\nu,\alpha_0^\vee\ra.\]
Hence, $p^s\la\gamma,\alpha_0^\vee\ra\leq
\la\sigma^*\gamma,\alpha_0^\vee\ra
\leq \la\nu,\alpha_0^\vee\rangle<p^s$. Hence $\gamma=0$.
 \end{proof}
 
 The following is an analogue of \cite[Prop.~5.3]{BNP06} for our situation. Our bounds on $s=[r/2]$ are rather worse than one would expect from exact analogue. This is because we have not been able, as in \cite[Lemma 5.2]{BNP06}, to show that the $|\Gamma|$ sections in $\G_\Omega(k)$ can be contracted to the smaller collection corresponding to $\Gamma_h$ used in [{\it ibid.}]. Ultimately this is due to a lack of the explicit determination of $\opH^1(G_{r/2},H^0(\lambda))$ made when $r$ is even in \cite{BNP04-Frob}.
 
 \begin{prop}\label{prop}Let $r\geq 19$, so that $s\geq 9$. Let $\lambda,\mu\in X_\sigma$ and $\Gamma'=\Gamma-\{0\}$.
 
 (a) We have
 \[\Ext^1_{G(\sigma)}(L(\lambda),L(\mu))\hookrightarrow \Ext^1_G(L(\lambda),L(\mu))\oplus R\] where \begin{align*}R&\cong\bigoplus_{\nu\in\Gamma'}\Ext^1_G(L(\lambda)\otimes V(\nu)^{[\sigma]},L(\mu)\otimes H^0(\nu))\\
 &\cong \bigoplus_{\nu\in \Gamma'}\Hom_{G/G_\sigma}(V(\nu)^{[\sigma]},\Ext^1_{G_\sigma}(L(\lambda),L(\mu)\otimes H^0(\nu)))\end{align*}
 
 (b) Let $5\leq t\leq s-4$. Set $\lambda=\lambda_0+p^t\lambda_1$ and $\mu=\mu_0+p^t\mu_1$ with $\lambda_0,\mu_0\in X_t$ and $\lambda_1,\mu_1\in X_{r/2-t}$. Then we may reidentify $R$ as
 \begin{align*}R&\cong \bigoplus_{\nu\in \Gamma'}\Ext^1_G(L(\lambda_1)\otimes V(\nu)^{[r/2-t]},L(\mu_1))\otimes \Hom_G(L(\lambda_0),L(\mu_0)\otimes H^0(\nu))\\&\cong
 \bigoplus_{\nu\in \Gamma'}\Hom_G(V(\nu)^{[r/2-t]},\Ext^1_{G_{r/2-t}}(L(\lambda_1),L(\mu_1)))\otimes \Hom_G(L(\lambda_0),L(\mu_0)\otimes H^0(\nu))\end{align*}
 \end{prop}
 
 \begin{proof} Aside from the last isomorphism, part (a) follows immediately from Theorem \ref{isom} and from the description of $\G_\Omega(k)$ given in Theorem \ref{filtration} and the remarks following. 
 
 For the last isomorphism, note that since $p^s\geq 2^{5}>2h-2=22$, the hypotheses of Lemma \ref{lemma} hold. Thus for any weight $\gamma\in \Gamma$ (which satisfies $\la\gamma,\alpha_0^\vee\ra<2h-2$) and any $\lambda,\mu\in X_\sigma$, we see that $\Hom_{G_\sigma}(L(\lambda),L(\mu)\otimes H^0(\gamma))$ has trivial $G$-structure. Now apply the LHS spectral sequence corresponding to $G_\sigma\triangleleft G$ to $E=\Ext^1_G(L(\lambda)\otimes V(\nu)^{[\sigma]},L(\mu)\otimes H^0(\nu))$. The $E_2^{1,0}$ term is \[\Ext^1_{G/G_\sigma}(V(\nu)^{[\sigma]},\Hom_{G_\sigma}(L(\lambda),L(\mu)\otimes H^0(\nu)))\cong \Ext^1_{G}(V(\nu),k)\otimes \Hom_{G_\sigma}(L(\lambda),L(\mu)\otimes H^0(\nu)),\] but $\Ext^1_{G}(V(\nu),k)=0$. So $E_2^{1,0}=0$. Similarly, $E_2^{2,0}=0$, so \[E\cong E_2^{0,1}=\bigoplus_{\nu\in \Gamma'}\Hom_{G/G_\sigma}(V(\nu)^{[\sigma]},\Ext^1_{G_\sigma}(L(\lambda),L(\mu)\otimes H^0(\nu)))\] and we get the identification of $R$ as given. 
 
For (b) write $\lambda$ and $\mu$ as suggested and, without loss of generality (dualising if necessary), assume $\la\mu_1,\alpha_0^\vee\ra<\la\lambda_1,\alpha_0^\vee\ra$. Use LHS corresponding to $G_t\triangleleft G$ on the terms in the first expression for $R$ in (a). We get
\begin{align*}\dim&\Ext^1_G(L(\lambda)\otimes V(\nu)^{[\sigma]},L(\mu)\otimes H^0(\nu))\leq \\
&\dim\Hom_{G/G_t}(L(\lambda_1)^{[t]}\otimes V(\nu)^{[\sigma]},\Ext^1_{G_t}(L(\lambda_0),L(\mu_0)\otimes H^0(\nu))\otimes L(\mu_1)^{[t]})\\
&+\dim\Ext^1_{G/G_t}(L(\lambda_1)^{[t]}\otimes V(\nu)^{[\sigma]},\Hom_{G_t}(L(\lambda_0),L(\mu_0)\otimes H^0(\nu))\otimes L(\mu_1)^{[t]})\end{align*}
Consider the first term, $E_2^{0,1}$, on the right-hand side. As $G_t$ is a classical Frobenius kernel, we have by \cite[(5.2.4)]{BNP06} that any weight $p^t\gamma$ of $\Ext^1_{G_t}(L(\lambda_0),L(\mu_0)\otimes H^0(\nu))$ satisfies $p^t\la\gamma,\alpha_0^\vee\ra<(p^t-1)h+1-\frac{p^t}{4}+\la\nu,\alpha_0^\vee\ra\leq p^s-\frac{5}{4}p^t+1+\la\nu,\alpha_0^\vee\ra<p^s-1+\la\nu,\alpha_0^\vee\ra$. (The second inequality holds since $p^t>p^4\geq h$ and $p^s\geq p^tp^4$.) Recall $V(\nu)^{[\sigma]}=(V(\nu)^{[1/2]})^{[s]}$. Now, if $\nu\neq 0$ the module $L(\lambda_1)^{[t]} \otimes V(\nu)^{[\sigma]}$ has a simple head with high weight $p^t\lambda_1+p^s\nu'$ for some $0\neq\nu'\in X^+$. Comparison of weights forces $E_2^{0,1}=0$.

Consider the other term $E_2^{1,0}$. Since $t\geq 5$, Lemma \ref{lemma} implies $\Hom_{G_t}(L(\lambda_0),L(\mu_0)\otimes H^0(\nu))$ has a trivial $G$-structure. Thus \begin{align*}\Ext^1_{G/G_t}&(L(\lambda_1)^{[t]}\otimes V(\nu)^{[\sigma]},\Hom_{G_t}(L(\lambda_0),L(\mu_0)\otimes H^0(\nu))\otimes L(\mu_1)^{[t]})\\
&=\Ext^1_{G/G_t}(L(\lambda_1)^{[t]}\otimes V(\nu)^{[\sigma]},L(\mu_1)^{[t]})\otimes\Hom_{G_t}(L(\lambda_0),L(\mu_0)\otimes H^0(\nu))\\
&=\Ext^1_{G}(L(\lambda_1)\otimes V(\nu)^{[r/2-t]},L(\mu_1))\otimes\Hom_{G_t}(L(\lambda_0),L(\mu_0)\otimes H^0(\nu))\end{align*} This shows the first reidentification.

For the second, we consider $\Ext^1_{G}(L(\lambda_1)\otimes V(\nu)^{[r/2-t]},L(\mu_1))$ and run the LHS spectral sequence corresponding to $G_{r/2-t}\triangleleft G$. The $E^{1,0}$ term is
\[\Ext^1_{G/G_{r/2-t}}(V(\nu)^{[r/2-t]},\Hom_{G_{r/2-t}}(L(\lambda_1),L(\mu_1)))\]

This is zero unless $\lambda_1=\mu_1$ and then it has trivial $G$-structure. Thus this term (and equally $E^{2,0}_2$) is trivial, and $\Ext^1_{G}(L(\lambda_1)\otimes V(\nu)^{[r/2-t]},L(\mu_1))\cong \Hom_{G/G_{r/2-t}}(V(\nu)^{[r/2-t]},\Ext^1_{G_{r/2-t}}(L(\lambda_1),L(\mu_1)))$, giving the other identification.\end{proof}
 
 \begin{corollary}\label{cor} With the hypotheses of the proposition, if either of the following hold:
 \begin{enumerate}\item $\Ext^1_{G_{r/2-t}}(L(\lambda_1),L(\mu_1))=0$;
 \item $\Hom_G(L(\lambda_0),L(\mu_0)\otimes H^0(\nu))=0$ for all $\nu\in \Gamma'$\end{enumerate}
 then $\Ext^1_{G(\sigma)}(L(\lambda),L(\mu))\cong \Ext^1_{G}(L(\lambda),L(\mu))$\end{corollary}
 
The following is an analogue of \cite[Thm.~5.4]{BNP06} showing that generically, self-extensions for the Ree groups between simple modules vanish.

\begin{theorem}\label{self}Let $r=2s+1$ be odd with $s\geq 9$. Then \[\Ext^1_{G(\sigma)}(L(\lambda),L(\lambda))=0\] for all $\lambda\in X_\sigma$.\end{theorem}
\begin{proof}Since $G$ is not of type $C_n$, we have $\Ext^1_{G_s}(L(\lambda),L(\lambda))$ for any $\lambda\in X_s$ by \cite[II.12.9]{Jan03}.

We wish to extend this result by replacing $s$ with $r/2$. When $r=1$ the result follows from \cite[1.7(1)(2),4.5]{Sin94}. Otherwise write $\lambda=\lambda_0+p^s\lambda_1$ with $\lambda_1\in X_{1/2}$ and apply the LHS spectral sequence corresponding to $G_s\triangleleft G_{r/2}$. The $E_2^{1,0}$ term is isomorphic to
\[\Ext^1_{G_{1/2}}(L(\lambda_1),\Hom_{G_s}(L(\lambda_0),L(\lambda_0))^{[-s]}\otimes L(\lambda_1))^{[s]}\cong \Ext^1_{G_{1/2}}(L(\lambda_1),L(\lambda_1))^{[s]}=0,\] by Sin's result. The $E_2^{0,1}$ term is isomorphic to \[\Hom_{G_{1/2}}(L(\lambda_1),\Ext^1_{G_s}(L(\lambda_0),L(\lambda_0))^{[-s]}\otimes L(\lambda_1))^{[s]}=0\] by \cite[II.12.9]{Jan03}.

The remainder of the argument in \cite[Thm.~5.4]{BNP06} goes through:
 Let $\lambda=\lambda_0+p^5\lambda_1$ with $\lambda_0\in X_5$ and $\lambda_1\in X_{r/2-5}$. Then by Proposition \ref{prop} (with $t=5$) it suffices to show that for all $\nu\in \Gamma'=\Gamma-\{0\}$  \[\Hom_G(V(\nu)^{[r/2-t]},\Ext^1_{G_{r/2-t}}(L(\lambda_1),L(\lambda_1)))\otimes \Hom_G(L(\lambda_0),L(\lambda_0)\otimes H^0(\nu))=0.\] By the above, we have $\Ext^1_{G_{r/2-t}}(L(\lambda_1),L(\lambda_1))=0$ so this follows.\end{proof}

Lastly we give a theorem relating $\Ext^1$s between simple $kG(\sigma)$-modules and $G$-modules.

\begin{theorem}\label{ext} Assume $r=2s+1$ with $s\geq 10$. Given $\lambda,\mu\in X_\sigma$, let \begin{align*}\lambda&=\sum_{i=0}^{r}(\tau^*)^i\lambda_{i/2}\\&=\lambda_0+\tau^*\lambda_{1/2}+p\lambda_1+p\tau^*\lambda_{3/2}+\dots+p^s\tau^*\lambda_{r/2}\end{align*} be the $\tau$-adic expansion of $\lambda$, and take a similar expression for $\mu$. Then there exists an integer $0\leq n<r$ such that \[\Ext^1_{G(\sigma)}(L(\lambda),L(\mu))\cong \Ext^1_G(L(\tilde\lambda),L(\tilde\mu))\]
where
\[\tilde\lambda=\sum_{i=0}^{n-1}(\tau^*)^i\lambda_{i/2+r/2-n/2}+\sum_{i=n}^{r}(\tau^*)^i\lambda_{i/2-n/2}.\]\end{theorem}
\begin{proof}If $\lambda=\mu$ the result follows from Theorem \ref{self} with $n=0$. Note that as $V^{[\sigma]}=V^{[r/2]}\cong_{G(\sigma)} V$ for any $G(\sigma)$ module $V$, it follows from Steinberg's tensor product theorem and examination of the $\tau$-adic expansion of $\lambda$ and $\tilde\lambda$ that $L(\tilde\lambda)\cong_{G(\sigma)} L(\lambda)^{[n/2]}$.  By \cite[2.1(c)]{SinF4} there is an injection $\Ext^1_G(L(\tilde\lambda),L(\tilde\mu))\hookrightarrow \Ext^1_{G(\sigma)}(L(\tilde\lambda),L(\tilde\mu))$ and since $\tau$ is an automorphism of $G(\sigma)$, we have $\Ext^1_{G(\sigma)}(L(\tilde\lambda),L(\tilde\mu))\cong \Ext^1_{G(\sigma)}(L(\lambda),L(\mu))$. Thus it suffices to show (by dimensions) that there is also an injection $\Ext^1_{G(\sigma)}(L(\tilde\lambda),L(\tilde\mu))\hookrightarrow \Ext^1_G(L(\tilde\lambda),L(\tilde\mu))$

Assume $\lambda\neq \mu$. Then there exists $0\leq i\leq r$ such that $\lambda_{i/2}\neq \mu_{i/2}$. Choose $n$ such that $\tilde\lambda_{5+1/2}\neq \tilde\mu_{5+1/2}$. Write $\tilde\lambda=\lambda'+p^{5+1/2}\lambda''+p^{6}\lambda'''$ with $\lambda'\in X_{5+1/2}$, $\lambda''\in X_{1/2}$ and $\lambda'''\in X_{r/2-6}$, and take a similar expression for $\mu$.

Since $5\leq 6\leq s-4$, Proposition \ref{prop} applies with $t=6$. Thus \[\Ext^1_{G(\sigma)}(L(\tilde\lambda),L(\tilde\mu))\hookrightarrow \Ext^1_G(L(\tilde\lambda),L(\tilde\mu))\oplus R,\] where $R$ is isomorphic to \[\bigoplus_{\nu\in \Gamma'}\Ext^1_G(L(\lambda''')\otimes V(\nu)^{[r/2-6]},L(\mu'''))\otimes \Hom_G(L(\lambda')\otimes L(\lambda'')^{[5+1/2]},L(\mu')\otimes L(\mu'')^{[5+1/2]}\otimes H^0(\nu)),\]where $\Gamma'=\Gamma-\{0\}$.

From Lemma \ref{lemma} we get 
\begin{align*}
\Hom_G(&L(\lambda')\otimes L(\lambda'')^{[5+1/2]},L(\mu')\otimes L(\mu'')^{[5+1/2]}\otimes H^0(\nu))\\
&\cong \Hom_{G/G_{5+1/2}}(L(\lambda'')^{[5+1/2]},\Hom_{G_{5+1/2}}(L(\lambda'),L(\mu')\otimes H^0(\nu))\otimes L(\mu'')^{[5+1/2]})\\
&\cong \Hom_{G}(L(\lambda''),L(\mu''))\otimes \Hom_{G_{5+1/2}}(L(\lambda'),L(\mu')\otimes H^0(\nu)).\end{align*}
Since $\lambda''\neq \mu''$ we get $R=0$ and so the theorem follows.
\end{proof}

\bibliographystyle{amsalpha}
\bibliography{bib}

\newcommand{\etalchar}[1]{$^{#1}$}
\providecommand{\bysame}{\leavevmode\hbox to3em{\hrulefill}\thinspace}
\providecommand{\MR}{\relax\ifhmode\unskip\space\fi MR }
\providecommand{\MRhref}[2]{%
  \href{http://www.ams.org/mathscinet-getitem?mr=#1}{#2}
}
\providecommand{\href}[2]{#2}
\begin{thebibliography}{BNP{\etalchar{+}}12}

\bibitem[BNP04a]{BNPtw}
Christopher~P. Bendel, Daniel~K. Nakano, and Cornelius Pillen, \emph{Extensions
  for finite groups of {L}ie type: twisted groups}, Finite groups 2003, Walter
  de Gruyter GmbH \& Co. KG, Berlin, 2004, pp.~29--46. \MR{2125064
  (2006a:20024)}

\bibitem[BNP04b]{BNP04-Frob}
\bysame, \emph{Extensions for {F}robenius kernels}, J. Algebra \textbf{272}
  (2004), no.~2, 476--511. \MR{2028069 (2004m:20089)}

\bibitem[BNP06]{BNP06}
\bysame, \emph{Extensions for finite groups of {L}ie type. {II}. {F}iltering
  the truncated induction functor}, Representations of algebraic groups,
  quantum groups, and {L}ie algebras, Contemp. Math., vol. 413, Amer. Math.
  Soc., Providence, RI, 2006, pp.~1--23. \MR{2262362 (2007i:20074)}

\bibitem[BNP{\etalchar{+}}12]{BNPPSS}
Christopher~P. Bendel, Daniel~K. Nakano, Brian Parshall, Cornelius Pillen,
  Leonard~L. Scott, and David.~I. Stewart, \emph{Bounding cohomology for finite
  groups and {F}robenius kernels}, arxiv:1208.6333 (2012).

\bibitem[Jan03]{Jan03}
Jens~Carsten Jantzen, \emph{Representations of algebraic groups}, second ed.,
  Mathematical Surveys and Monographs, vol. 107, American Mathematical Society,
  Providence, RI, 2003. \MR{MR2015057 (2004h:20061)}

\bibitem[Sin92]{Sin92}
Peter Sin, \emph{Extensions of simple modules for {${\rm Sp}_4(2^n)$} and
  {${\rm Suz}(2^m)$}}, Bull. London Math. Soc. \textbf{24} (1992), no.~2,
  159--164. \MR{1148676 (93b:20025)}

\bibitem[Sin93]{Sin93}
\bysame, \emph{Extensions of simple modules for {$G_2(3^n)$} and
  {${}^2G_2(3^m)$}}, Proc. London Math. Soc. (3) \textbf{66} (1993), no.~2,
  327--357. \MR{1199070 (94a:20020)}

\bibitem[Sin94a]{SinF4}
\bysame, \emph{The cohomology in degree {$1$} of the group {$F_4$} in
  characteristic {$2$} with coefficients in a simple module}, J. Algebra
  \textbf{164} (1994), no.~3, 695--717. \MR{1272111 (95c:20063)}

\bibitem[Sin94b]{Sin94}
\bysame, \emph{Extensions of simple modules for special algebraic groups}, J.
  Algebra \textbf{170} (1994), no.~3, 1011--1034. \MR{1305273 (95i:20066)}

\end{thebibliography}

\end{document}